\documentclass{amsart}
\usepackage{fullpage}

 \usepackage{graphicx}

\usepackage{amssymb, amsmath}

\usepackage{amssymb, amsmath}
\usepackage[usenames,dvipsnames]{xcolor}
\usepackage{color,ulem,textcomp}

\usepackage[english,francais]{babel}

\newtheorem{thm}{Theorem}[section]
\newtheorem{lem}[thm]{Lemma}
\newtheorem{prop}[thm]{Proposition}
\newtheorem{e-proposition}[thm]{Proposition}
\newtheorem{corollary}[thm]{Corollary}
\newtheorem{e-definition}[thm]{Definition\rm}
\newtheorem{rmk}{\it Remark\/}

\def\R {\mathbb{R}}

\def\cA {\mathcal{A}}

\def\cD {\mathcal{D}}

\def\cL {\mathcal{L}}

\def\cQ {\mathcal{Q}}

\def\cZ {\mathcal{Z}}

\def\eps {{\varepsilon}}

\def\indc {{\bf 1}}

\def\d {{\partial}}

\newcommand{\Ker}{\operatorname{Ker}}
\newcommand{\Kn}{\operatorname{Kn}}
\newcommand{\Ma}{\operatorname{Ma}}
\newcommand{\Rn}{\operatorname{Re}}

\def\og{\leavevmode\raise.3ex\hbox{$\scriptscriptstyle\langle\!\langle$~}}
\def\fg{\leavevmode\raise.3ex\hbox{~$\!\scriptscriptstyle\,\rangle\!\rangle$}}

\begin{document}
\centerline{Comptes Rendus Physique. Acad\'emie des Sciences. Paris}


\selectlanguage{english}
\title{A microscopic view on the Fourier law}


\selectlanguage{english}

\author{Thierry Bodineau}
\email{thierry.bodineau@polytechnique.edu}
\address{CMAP, Ecole polytechnique}

\author{Isabelle Gallagher}
\email{isabelle.gallagher@ens.fr}
\address{DMA, Ecole normale sup\'erieure}

\author{Laure Saint-Raymond}
\email{laure.saint-raymond@ens-lyon.fr}
\address{UMPA, Ecole normale sup\'erieure de Lyon}


\maketitle

\medskip

\begin{center}
{\small Received *****; accepted after revision +++++\\
Presented by £££££}
\end{center}

\begin{abstract}
\selectlanguage{english}
The  Fourier law of heat conduction describes  heat diffusion in macroscopic systems.
This physical law has been experimentally tested  for a large class of physical systems. 
A natural question is   to know whether it can be derived from the microscopic models using the fundamental laws of mechanics.
\vskip 0.5\baselineskip

\selectlanguage{french}
\noindent{\sc R\'esum\'e.\ } 
La loi de Fourier permet de d\'ecrire la diffusion de la chaleur dans des syst\`emes physiques \`a l'\'echelle
 macroscopique. Cette loi  est tr\`es bien v\'erifi\'ee exp\'erimentalement et
une question naturelle est   de  la justifier \`a partir de mod\`eles microscopiques en utilisant les principes fondamentaux de la m\'ecanique.
\end{abstract}


\selectlanguage{english}

\section{The Fourier law~: a phenomenological model?}

The {\it Fourier law} is a physical law
which relates  the local thermal flux $J$ to the small variations of temperature $\nabla T$  in a domain $\Omega$
\begin{equation}
\label{eq: local Fourier law}
J = - \kappa (T) \nabla T\, ,
\end{equation}
where $\kappa(T)$ is  a non negative~$3\times3$ matrix, referred to as the thermal conductivity of the material. This relation is empirical, but is satisfied approximately in most physical systems.
>From the Fourier law \eqref{eq: local Fourier law}, one can then deduce  a diffusion equation for the temperature 
$$c(T) \d_t T =  \nabla\cdot (\kappa (T) \nabla T)\,,$$
where $c(T) $ is the specific heat of the system. This is the famous {\it heat equation} for which Fourier developed a number of mathematical tools.

All along the nineteenth century similar laws have been discovered such as Ohm's law for electric currents or Fick's law for particle currents. One  challenge   in physics, also addressed to mathematicians by Hilbert in his sixth problem, is to understand these macroscopic phenomenological laws starting from the first principles, i.e. from microscopic models of interacting atoms.

At the microscopic scale, matter is made of elementary particles, the evolution of which is governed by the classical laws of mechanics. In this framework, the goal of statistical physics   to model heat conduction is twofold~: 
\begin{itemize}
\item to give a thermodynamic definition of macroscopic quantities (internal energy,  temperature, pressure...) as statistical parameters   of the microscopic system;
\item to deduce a rigorous derivation of the constitutive laws such as  the Fourier law.
\end{itemize}
The main difficulty is that   thermodynamic quantities as temperature are well defined at equilibrium, instead the Fourier law describes a non-equilibrium phenomenon. 
Thus one has to derive an appropriate notion of local equilibrium which allows to decouple the fast and slow modes of the system and to define a local version of the temperature.
Such property of the system is related to ergodicity and mixing phenomena. From the mathematical point of view, establishing such a property for purely deterministic evolutions is a 
mathematical challenge as pointed out in \cite{BLR}.

In this note, we try to investigate this question for  models of gases very close to equilibrium, namely first for perfect gases which correspond already to an idealization of  dilute gases (Section 2) and then for hard-sphere gases in the Boltzmann-Grad limit (Section 3).


\section{The case of perfect gases}

\subsection{Diffusive limits of the Boltzmann equation}$ $

Out of thermodynamic equilibrium, perfect monatomic gases in a smooth domain~$\Omega$ of the $d$ dimensional  space~$\R^d$ can be described by their distribution~$f \equiv f(t,x,v)$ which is  the probability of finding a particle with position $x\in \Omega$ and velocity $v\in \R^d$ at time $t\in \R$.
The evolution of this probability density is governed by the Boltzmann equation \cite{boltzmann}~:
$$\d_t f+ v\cdot \nabla_x f = Q(f,f)$$
where the left-hand side expresses  the free transport of particles, while the right-hand side encodes the statistical effect of pointwise elastic collisions
$$ Q(f,f) (v) := \int_{ {\mathbb R}^d \times {\mathbb S}^{d-1}} \big (f(v'_*) f(v') - f(v_*)f(v)\big) b(v-v_*,\omega) dv_* d\omega\,.$$
The pre-collisional velocities $v', v'_*$ appearing in the ``gain term" are defined in terms of  the post-collisional velocities and of the deflection parameter $\omega\in  {\mathbb S}^{d-1}$~:
$$ v'= v- (v-v_*) \cdot \omega \; \omega \, , \quad v'_* = v_*+ (v-v_*) \cdot \omega \; \omega$$
so that the momentum and kinetic energy are preserved~: $v'+v'_* = v+v_*$ and  $|v'|^2+|v'_*|^2 = |v|^2+|v_*|^2$.

The collision cross-section $b$ gives the statistical law of this (velocity) jump process. The exchangeability of particles and the micro-reversibility of the collision process
 imply that  $b$ depends only on $|v-v_*|$ and on~$|(v-v_*)\cdot \omega|$. The hard-sphere interaction corresponds to $b(v-v_*, \omega) = \big( (v-v_*)\cdot \omega\big)_+$. 
The Boltzmann equation describes the typical behavior of a gas evolving according to Newtonian dynamics.
We shall discuss this \underline{``perfect gas" approximation}  in the next section.
 
In the case when the domain~$\Omega$ has boundaries, the Boltzmann equation must be supplemented with boundary conditions which will be specified below, see~(\ref{Req}).
 
\medskip
Thermodynamic variables, such as the temperature $T$ {\color{black} or the mass density $R$}, are defined as averages of $f$ with respect to the~$v$-variable. More precisely, 
$T$ is related to the variance of $f$ since it measures the 
thermal agitation, or equivalently  the dispersion  of particle velocities around the bulk velocity  
{\color{black} $$ \begin{aligned}
&  R(t,x) := \int f(t,x,v)  dv, \qquad RU(t,x) := \int  f(t,x,v) v dv \, , \\
 & T(t,x) := \frac{1}{R(t,x) } \left(\frac{1}{d} \int f(t,x,v) \big(v- U(t,x)\big)^2 dv\right)  \, .
 \end{aligned}$$
 }
 In general, we do not expect these variables to satisfy a closed system of equations. This will be the case only at  local thermodynamic equilibrium, i.e. in the limit when the collision process is much faster than the transport.
 The accuracy of this approximation is therefore measured by the Knudsen number $\Kn$, which is the ratio between the mean free path and the typical observation length.
 
 For perfect gases, as the elementary particles have zero volume, this Knudsen number is related to two other key non-dimensional parameters by the Von Karman relation
 $ \Kn = \Ma / \Rn$, where the Mach number~$\Ma$  measures the compressibility of the gas and is defined as the ratio of the bulk velocity to the thermal speed, and the Reynolds number $\Rn$  measures the kinematic viscosity of the gas.
 Diffusive limits are then obtained in regimes where both the Knudsen and Mach numbers are small, say of order $\alpha\ll1$. After a suitable non dimensionalisation   (and in particular rescaling time to~$t/\alpha$ in order to recover a diffusive regime), the Boltzmann equation can be restated
 \begin{equation}
 \label{Beq}
 \alpha \d_t f_\alpha + v\cdot \nabla_x f_\alpha = \frac1\alpha Q(f_\alpha,f_\alpha)\,.
 \end{equation}
 
 In order for these scaling assumptions to be consistent, the distribution $f_\alpha$ has to satisfy  that  the bulk velocity $U_\alpha $ is very small (of order $\alpha^\gamma $ with $\gamma \geq 1$) compared to the   thermal velocity $\sqrt{T_\alpha}$.  
 This is typically the case when $f_\alpha$ is a small  perturbation  (of order $O(\alpha^\gamma) $) around a thermodynamic equilibrium $M= M(v)$, i.e. a stationary solution to the Boltzmann equation $Q(M,M) = 0$. It could be also the case in more general situations, but  it is not known how to obtain an a priori control on the Mach number. For a perturbation  of the form
 $$
 f_\alpha = M(1+\alpha^\gamma g_\alpha)\, ,
 $$
the temperature is $T_\alpha = T+O(\alpha^\gamma)$ and we shall be interested in describing the evolution of the temperature   in this diffusive regime. Without loss of generality, we can assume that $M$ is the unit 
centered Gaussian:~$\displaystyle 
 M(v):=\frac{1}{(2\pi)^{\frac d2}} e^{-|v|^2/2}.
 $
 
\medskip

Now let us prescribe boundary conditions to~(\ref{Beq}). 
In order to observe heat fluxes, we impose a diffuse reflection on the boundary, the temperature of which is fixed. 
As the Boltzmann equation~(\ref{Beq})  is a transport equation, only the incoming fluxes have to be prescribed: 
 denoting~$n(x)$ the outward unit normal vector   at~$x\in \partial\Omega$, we define
$\displaystyle\Sigma_\pm := \big\{ (x,v) \in \d \Omega \times {\mathbb R}^3 \,/\, \pm v\cdot n(x)  >0\big\}$
and boundary conditions have to be prescribed on~$\Sigma_-$:  we impose
\begin{equation}
\label{Req}
\begin{aligned}
 f_\alpha (t,x, v) &=  M_{\Sigma}(v) \int_{ w \cdot n(x)>0} f_\alpha (t,x,w)   ( w \cdot n(x))_+ \,  d  w \, \,  \hbox{ on }  \, \, \Sigma_-\,,
  \end{aligned}
  \end{equation}
  where~$M_{\Sigma}(v) $ is a fixed isotropic function such that
$$
\int M_{\Sigma}(v)  (v  \cdot n(x))_+  dv 
= 1 \quad \mbox{and} \quad \int M_{\Sigma}(v) |v|^2 (v  \cdot n(x))_+  dv = (d+1) T\, .
$$
Thus the Maxwellian~$M_{\Sigma}(v)$ imposing the temperature $T$ at the boundary has the form
\begin{equation}
\label{eq: maxwell bord}
M_{\Sigma} (v):=\frac{1}{(2\pi T)^{\frac d2}} \sqrt{\frac{2\pi }{T}} \exp \left( - \frac{|v|^2}{2 T} \right).
\end{equation}
Note that this ensures in particular that there is no mass flux at the boundaries:
$$\int f_\alpha (t,x,v)  \, (v \cdot n(x) ) \, dv = 0 \, \, \hbox{ on } \, \,\partial \Omega\,.$$
In accordance with the scaling assumptions, the  shift of the temperature at the boundary has to be  of order~$O(\alpha^\gamma )$. 
{\color{black} Denoting by  $\bar \theta_{|\partial\Omega}$  the small variations, at the scale $\alpha^\gamma$, of the 
 temperature profile on the boundary $\partial\Omega$ of the domain, we set}
$T_{|\partial\Omega}= 1+\alpha^\gamma  \bar \theta_{|\partial\Omega}$. 
Linearizing \eqref{eq: maxwell bord} for small $\alpha$, we set
$$
 M_{\Sigma}(v) :=\sqrt{2\pi} M(v) \Big(1+\alpha^\gamma  \bar \theta_{|\partial\Omega} \frac{|v|^2- (d+1)}{2}\Big)\, ,
$$
which is just the linearisation of \eqref{eq: maxwell bord}.
Notice that
\begin{equation}
\label{eq: gaussian moments}
 \int M(v)  \frac{|v|^2- (d+1)}{2} (v  \cdot n(x))_+  dv 
 =   \int M(v) \frac{|v|^2-(d+1)}{2} |v|^2 (v  \cdot n(x))_+  dv =0\, .
\end{equation}
We shall now study the limit of (\ref{Beq}), (\ref{Req}) as $\alpha \to 0$. Of course the parameter $\alpha$ is a physical quantity which is fixed, but we expect the limit to be a good approximation of the solution when $\alpha \ll1$. Moreover, we shall use a \underline{weak notion of convergence filtering out oscillations} on small (spatial/temporal) scales since they are not relevant for the macroscopic description.

\subsection{The heat equation as  scaling limit of the local conservation of energy}


Let us first explain the formal asymptotics. 
The variation of temperature at the boundary modifies the density in the bulk as~$f_\alpha= M(1+\alpha^\gamma g_\alpha)$. Plugging this Ansatz  in (\ref{Beq}), we get
\begin{equation}
\label{Beq2}
\alpha \d_t g_\alpha + v\cdot \nabla_x g_\alpha = - \frac1\alpha \cL g_\alpha +\alpha^{\gamma-1} Q (g_\alpha,g_\alpha)\,,
\end{equation}
where 
$$ \cL h (v) := \int \big( h(v_*)+h(v) - h(v'_*) - h(v') \big)M(v_*)  b(v-v_*,\omega) dv_* d\omega\,.$$

\medskip
\noindent
\emph{Step 1.} 
Provided that one can obtain a uniform a priori bound on $g_\alpha$ (in a suitable functional space 
such that all the terms are well defined), we deduce that any limit point $\bar g$ of the sequence $(g_\alpha)$ satisfies
$\cL \bar g = 0$.
Using the symmetries of~$\cL$ inherited from the exchangeability of particles $(v,v_*) \mapsto (v_*, v)$ and from the micro-reversibility of collisions $(v,v_*, \omega) \mapsto (v',v'_*,\omega)$, we obtain that $\cL h = 0 $ if and only if 
$$ \int Mh \cL h  dv= \frac14 \iint \big( h(v_*)+h(v) - h(v'_*) - h(v') \big)^2 M(v) M(v_*)  b(v-v_*,\omega) dvdv_* d\omega= 0\,,$$
which implies that 
$ h(v)+h(v_*) - h(v')-h(v'_*) = 0$ for almost all $(v,v_*,\omega).$
One can then prove (using a suitable regularization and differentiating this relation with respect to $\omega$) that~$h$ is a linear combination of the collision invariants $1, v, |v|^2$ (cf. \cite{grad}). 
We conclude that $\bar g$ has to be an   infinitesimal Maxwellian 
\begin{equation}
\label{thermo-eq}
\bar g (t,x,v) =  \rho(t,x) + u(t,x) \cdot v + \theta(t,x) {|v|^2 - d \over 2} \, ,
\end{equation}
{\color{black} where $\rho,u,\theta$ stand for the density, momentum and energy profiles.}

\medskip
\noindent
\emph{Step 2.} The conservation laws associated with (\ref{Beq2}) state
$$ \alpha \d_t \int  g_\alpha \begin{pmatrix} 1\\v\\|v|^2\end{pmatrix} M dv +\nabla_x \cdot  \int    vg_\alpha \begin{pmatrix} 1\\v\\|v|^2\end{pmatrix} M dv = 0\,.$$
>From the conservation of mass, we deduce the incompressibility constraint $$\nabla_x  \cdot u  = \nabla \cdot \int v \bar g Mdv  = 0\, , $$ and from the conservation of momentum we get the Boussinesq relation
$$ \d_i  (\rho+\theta )  = \d_i  \int v_i^2 \bar g Mdv  =0\, .$$

\medskip
\noindent
\emph{Step 3.} For the next order in the expansion,  we use the crucial fact that $\cL$ is a self-adjoint Fredholm operator with kernel spanned by $1, v_i,|v|^2$. Since $\psi(v) := v \big( |v|^2- (d+2) \big)$ and 
$\phi(v) := v\otimes v - \frac{1}{d} |v|^2  \text{Id}$ belong to~$(\Ker \cL)^\perp$, we have
$$
\begin{aligned}
&\d_t \int Mg_\alpha v dv +\frac1\alpha \nabla_x \cdot \int \cL g_\alpha \cL^{-1} \phi M dv 
+\frac1{d \, \alpha} \nabla_x \int g_\alpha |v|^2 M dv = 0\, ,\\
&\d_t \int Mg_\alpha \big( |v|^2- (d+2) \big)  dv 
+\frac1\alpha \nabla_x \cdot \int \cL g_\alpha \cL^{-1} \psi M dv  = 0\,.
\end{aligned}
$$
Replacing $\alpha^{-1}\cL g_\alpha$ with (\ref{Beq2}), we obtain only terms of order $O(1)$ up to a gradient term.

- If $\gamma>1$, the nonlinear terms vanish in the limit, and we keep only the diffusion
$$ \d_t u -\nu \Delta _x u +\nabla_x p = 0\, , \quad \d_t \theta -\kappa \Delta _x \theta = 0\,,$$
where the second equation is exactly the Fourier law.

- If $\gamma = 1$,  using the identity $\cQ (\bar g,\bar g) = \frac12 \cL (\bar g^2)$ (which holds true for any  infinitesimal Maxwellian $\bar g$), one can show (see \cite{BGL1})  that the nonlinear terms correspond to the convection~:
$$ \d_t u +u \cdot \nabla_x u -\nu \Delta _x u +\nabla_x p = 0\, , \quad \d_t \theta+ u\cdot \nabla_x \theta  -\kappa \Delta _x \theta = 0\,.
$$

\medskip
\noindent
\emph{Step 4.} 
At leading order in $\alpha$, the boundary condition (\ref{Req}) states 
\begin{equation}
\label{eq: boundary terms g}
g_\alpha(t,x,v) 
=  \bar \theta_{|\partial\Omega} {|v|^2- (d+1) \over 2} +\int  g_\alpha (t,x,w)  \sqrt{2\pi} M(w)( w  \cdot n(x))_+   d w  \quad \hbox{ on }\quad  \Sigma_-\,.
\end{equation}
If the moments $\rho,u,\theta$ are smooth enough, then we can take the trace  of the constraint equation (\ref{thermo-eq}), and plug this Ansatz in the boundary condition~:
$$ 
\rho +u \cdot v +\theta {|v|^2 - d \over 2} 
=   \bar \theta_{|\partial\Omega}  {|v|^2-  (d+1)\over 2} + (\rho+ u\cdot n\sqrt{2\pi} + \frac12 \theta) \quad \hbox{ on }\quad  \Sigma_- \,, 
$$
from which we deduce that $u= 0$ and $\theta=  \bar \theta_{|\partial\Omega} $ on the boundary.

\medskip
\noindent
This formal analysis relies on many assumptions which need to be checked
\begin{itemize}
\item[(i)] the uniform a priori bounds on $(g_\alpha)$;
\item[(ii)] the kinetic equation (\ref{Beq2}) and conservation laws satisfied in some  suitable functional spaces;
\item[(iii)] the stability of the nonlinear (convection) terms;
\item[(iv)] the regularity of the moments up to the boundary.
\end{itemize}

\medskip
For the sake of simplicity, we shall assume from now on that the  non-equilibrium pertubation 
is very small ($\gamma>1$) so that the nonlinear effects (namely the convection terms) are higher order corrections. 
>From the mathematical point of view, we shall consider the linearized Boltzmann equation
 \begin{equation}
 \label{LBeq}
\alpha \d_t g_\alpha + v\cdot \nabla_x g_\alpha = -\frac1\alpha \cL g_\alpha 
\end{equation}
supplemented with the boundary conditions \eqref{eq: boundary terms g}
which will induce two major simplifications~: 
\begin{itemize}
\item First of all, the functional space defined by energy inequality  (which is actually the linearized version of the physical entropy inequality)
\begin{equation}
\label{Lentropy}
\begin{aligned}
\frac12& \int M g_\alpha^2(t,x,v) dxdv +\frac1{\alpha^2} \int_0^t \int Mg_\alpha \cL g_\alpha (s,x,v) dxdvds\\
 & \leq \frac12 \int M g_{\alpha,0}^2(t,x,v) dxdv -\frac1{2\alpha}  \int_0^t \int_{\d\Omega \times {\mathbb R}^3} M g_\alpha^2(t,x,v) (v\cdot n (x) ) d\sigma_x dv
  \end{aligned}
\end{equation}
is built on $L^2$ (with weights).
In the inequality above, $d\sigma_x$ is the surface measure on $\d \Omega$.
We therefore have a Hilbertian structure (and there is no need to renormalize the kinetic equation).
\item
Moreover, there are no nonlinear terms in the local conservations laws, which makes it easier to study the limit as $\alpha \to 0$. (Note that for the nonlinear Boltzmann equation, it is not even known whether the local conservation laws are satisfied for fixed $\alpha>0$.)
\end{itemize}
We shall therefore focus on points (i) and (iv), and  establish the following result.
\begin{thm}
Consider a perfect gas with Knudsen and Mach numbers of order $\alpha$
evolving according to \eqref{LBeq} with the boundary conditions \eqref{eq: boundary terms g} 
in a smooth domain~$\Omega$ at a given temperature $\bar\theta_{|\partial\Omega}$.
 Then, in the  diffusive limit  $\alpha \to 0$, the density fluctuation $g_\alpha $ converges weakly  to the infinitesimal Maxwellian
$$g(t,x,v):=  u (t,x) \cdot v + \frac{|v|^2 - (d+2)}2 \theta (t,x) \, , $$
where~$(u,\theta)$ satisfies the Stokes-Fourier equations
$$\begin{aligned}
&\partial_t u -\nu  \Delta_x u = - \nabla_x p \,  , \quad \nabla_x \cdot u = 0\, , \quad u_{|\d\Omega } = 0\, ,\\
&\partial_t \theta -\kappa  \Delta_x \theta = 0\, , \quad \theta_{|\d\Omega} =\bar\theta_{|\partial\Omega}\,. \\
\end{aligned}
$$
\end{thm}

We stress the fact that more complete results have been established starting from the full Boltzmann equation.
Refined tools such as hypoellipticity, averaging lemmas and compensated compactness allow actually to study the limit of \eqref{Beq},\eqref{Req} as $\alpha \to 0$ for any $\gamma \geq 1$ in the framework of renormalized solutions, provided that $\bar\theta_{|\partial\Omega} = 0$ (see \cite{BGL1}, \cite{BGL2}, \cite{LM}, \cite{GSR}, \cite{MSR},  \cite{SR-LNM} and references therein). Extending these results for inhomogeneous boundary conditions, i.e. when $\bar\theta_{|\partial\Omega}\neq 0$, seems however to be very difficult since a uniform a priori bound (that would be the counterpart of \eqref{MLentropy-bound} in the nonlinear setting) is missing.

There are  also alternative approaches dealing with smooth  solutions to the Boltzmann equation, based on energy estimates in weighted Sobolev spaces (see \cite{BU}, \cite{LYY} in the time dependent case, and  \cite{ELM}, \cite{EGKM1}, \cite{EGKM2} in the stationary case). 
In particular, the case of non-homogenous boundary temperature has been addressed  in \cite{ELM2} (see also 
 \cite{ELM}) and a quantitative convergence  towards the compressible Navier-Stokes equation has been established in a stationary regime.

It is not clear  whether the sophisticated methods described above could be used at the level of particles.
Our result is  more modest as it holds only in a linear regime, however its derivation is much simpler and we hope that it can be helpful for future applications to particle systems.

\subsection{First step of the proof~: establishing uniform a priori bounds}
\label{sec: establishing uniform a priori bounds}

The first difficulty we encounter to justify the formal derivation of the previous paragraph  is that the energy inequality (\ref{Lentropy})  has a source term because of the boundary condition 
which seems to be of order $O(1/\alpha)$. It is then not clear that the sequence~$(g_\alpha)$ is uniformly bounded in $L^2$, and that we can extract converging subsequences.

When $\bar\theta_{|\partial\Omega} = 0$, this term (which is the linearized version of the Darroz\`es-Guiraud information \cite{DG})  has a sign.
To see this, let us first recall the definition of the boundary 
$$
\Sigma_\pm := \big\{ (x,v) \in \d \Omega \times {\mathbb R}^3 \,/\, \pm v\cdot n(x)  >0\big\}
$$
and denote by  $d\mu_x(v) :=\sqrt{2\pi} M (v\cdot n(x))_+ dv$ the probability measure on the boundary.
Decomposing $g_\alpha$ on~$\Sigma_+$ and $\Sigma_-$ (by using \eqref{eq: boundary terms g} with 
$\bar\theta_{|\partial\Omega} = 0$), we deduce that
\begin{align}
 \int_\Sigma g_\alpha^2(t,x,v) \sqrt{2\pi} M (v\cdot n(x)) dv d\sigma_x  
 & = \int_{\Sigma_+}  g_\alpha^2(t,x,v) d\mu_x(v) d\sigma_x  - \int_{\Sigma_+}  \Big( \int_{\Sigma_+}   g_\alpha d\mu_x(w) \Big)^2  d\mu_x(v) d\sigma_x
 \nonumber \\
 & = \int_{\Sigma_+}  \left( g_\alpha (t,x,v)    -   \int_{\Sigma_+}   g_\alpha d\mu_x(w) \right)^2  d\mu_x(v) d\sigma_x 
 \geq 0\, .
 \label{eq: variance bord}
\end{align}

Then, we obtain directly  a uniform  estimate on $(g_\alpha)$ in $L^\infty_t (L^2(dxMdv)) $, a control on the relaxation 
\cite{grad}
$$
\frac1{\alpha^2} \| g_\alpha - \Pi g_\alpha\|^2_{L^2 (dtdx a Mdv )} \leq \frac1{\alpha^2}  \int_0^t \int g_\alpha \cL g_\alpha Mdvdxds \leq C_0\,  ,
$$
where   $a$ is   the collision frequency:  $a(|v|) :=\displaystyle \int M(v_*)b(v-v_*,\omega) dv_* d\omega$,
and $\Pi$ stands for the projection onto~$\Ker \cL$.
We also have a control at the boundary
$$ 
\frac1\alpha \int_{\Sigma_+} \Big( g_\alpha - \int g_\alpha d\mu_x (v)\Big) ^2 d\mu_x(v) d\sigma_x ds\leq C_0 \,.
$$
Note that this last control implies more or less that $u_\alpha$ and~$\theta_\alpha$ have to vanish at the boundary (Dirichlet boundary condition) since $g_\alpha$ should look like its average with respect to $d\mu_x(v)$.
In particular, we cannot hope to prove a similar result in the case when $\bar \theta_{|\partial\Omega}\ne 0$. In other words, it is not a good idea to control the fluctuation $g_\alpha$ with respect to the Maxwellian $M$, since $M$ does not satisfy the right boundary conditions and therefore is not an (approximate) equilibrium for (\ref{LBeq})
with the boundary conditions \eqref{eq: boundary terms g}.

\medskip

The key idea here is then to introduce a \emph{modulated entropy} to measure the fluctuation with respect to a Maxwellian which has a suitable  behaviour at the boundary. More precisely, we extend the temperature to the whole domain~$\Omega$ by defining  a smooth function~$\tilde \theta$ such that 
$ \tilde  \theta_{|\partial\Omega} = \bar \theta_{|\partial\Omega} $
and the following modulated fluctuation
\begin{equation}
\label{eq: changement mesure}
\tilde g_\alpha (t,x,v):=  g_\alpha (t,x,v) - \tilde \theta(x) {|v|^2 - (d+2)\over 2} \,\cdotp
\end{equation}
A similar idea was used to obtain a priori estimates  in different contexts involving boundaries, namely for hydrodynamic limits of lattice systems   (see for example \cite{ELS}) or for singular perturbation problems as in \cite{DSR}. Note that this is  different from \cite{Y} or \cite{LYY} where the modulated entropy/energy is used directly to prove the convergence (which requires therefore to construct a precise approximate solution).
\begin{prop}
With the previous notation, the modulated entropy satisfies the uniform bound
\begin{equation}
\label{MLentropy-bound}
\begin{aligned}
 \int M \tilde g_\alpha^2(t,x,v) dxdv &+\frac1{\alpha^2} \int_0^t \int M\tilde g_\alpha \cL \tilde g_\alpha (s,x,v) dxdvds  \\&\quad +
 \frac{1}{\alpha   \sqrt{2 \pi}}\int_0^t  \int_{\Sigma_+} \Big( \tilde g_\alpha (s,x,v) - \int \tilde g_\alpha (s,x,w) d\mu_x (w)\Big) ^2  d\mu_x(v) d\sigma_x ds\leq C(t) \, .
  \end{aligned}
\end{equation}
\end{prop}

\begin{proof}
The modulated fluctuation \eqref{eq: changement mesure} has been chosen so that 
the system at the boundary can be compared to a local equilibrium as in \eqref{Req}. 
As a consequence, the boundary conditions \eqref{eq: boundary terms g}  can be rewritten   for any $x$  in $\Sigma_-$ as
\begin{align}
\tilde  g_\alpha (t,x, v) 
&=  \frac12   \bar \theta_{|\partial\Omega} +\int  g_\alpha (t,x,w) d\mu_x(w)  dw 
= \int \Big( g_\alpha  (t,x,v')- \tilde \theta(x) {|v'|^2  - (d+2) \over 2}\Big) d\mu_x(w)  dw 
\nonumber \\ 
&= \int  \tilde g_\alpha (t,x,w) d\mu_x(w)\, ,
\label{eq: bord tilde}
\end{align}
where we used the Gaussian moments  \eqref{eq: gaussian moments} at the boundary.
Adjusting the boundary conditions comes at a  price and the evolution equation \eqref{LBeq}
has to be replaced by 
\begin{equation}
\label{LBeq 2}
\alpha \d_t \tilde g_\alpha + v\cdot \nabla_x \tilde g_\alpha = -\frac1\alpha \cL \tilde g_\alpha 
+  v\cdot \nabla_x   \tilde \theta(x) \, {|v|^2 - (d+2)\over 2}\,,
\end{equation}
where we used  the local conservations of mass and energy  so that   
${\mathcal L}\tilde g_\alpha = {\mathcal L}  g_\alpha$.

The same computation as for the Darroz\`es-Guiraud inequality applies for $\tilde g_\alpha$ with an additional contribution in the bulk due to the modulated fluctuation
 \begin{equation}
\label{MLentropy}
\begin{aligned}
& \int M \tilde g_\alpha^2(t,x,v) dxdv 
+ \frac1{\alpha^2} \int_0^t \int M  \, \tilde g_\alpha \cL  \, \tilde g_\alpha (s,x,v) dxdvds
+  \frac1{\alpha}  \int_0^t \int_\Sigma M \tilde g_\alpha ^2 (s,x,v)\,  ( v\cdot n (x)) \,   d\sigma_x dvds \\
 &\qquad =   \int M \tilde g_{\alpha,0}^2(t,x,v) dxdv 
 + \frac1\alpha \int_0^t \int \nabla_x \tilde \theta(x) \cdot v 
 \frac{|v|^2 - (d+2)}{2} \, \tilde g_\alpha (s,x,v)  M(v)dvdx  ds \, .
 \end{aligned}
\end{equation}
%
The boundary term in (\ref{MLentropy-bound}) can been obtained  as in 
\eqref{eq: variance bord} thanks to the tilted  boundary conditions \eqref{eq: bord tilde} 
\begin{equation*}
\sqrt{2 \pi} \int_0^t \int_\Sigma M \tilde g_\alpha ^2  (s,x,v)\, ( v\cdot n (x)) \,  d\sigma_x dv ds 
=
\int_0^t  \int_{\Sigma_+} \Big( \tilde g_\alpha (s,x,v)  - \int \tilde g_\alpha (s,x,w)  d\mu_x (w)\Big) ^2 d\mu_x(v) d\sigma_x ds\, .
\end{equation*}
It then remains  to control the flux term. Using again $\psi(v) := v \big( |v|^2- (d+2) \big)$, we get 
\begin{align*}
&\Big| \frac1\alpha \int_0^t \int \nabla_x \tilde \theta \cdot v(|v|^2 - (d+2)) \, \tilde g_\alpha Mdvdx ds \Big| 
= \Big|\frac1\alpha \int_0^t \int \nabla_x \tilde \theta \cdot \cL^{-1} \psi \cL  \, \tilde  g_\alpha Mdvdx ds\Big|\\
&\qquad \qquad 
\leq C \sqrt{t}  \| \nabla \tilde \theta \|_{L^2(dx)}   \| \cL^{-1} \psi\| _{L^2 (M adv)} \left\|\frac1\alpha \cL \tilde  g_\alpha \right\|_{L^2 (dt dx Ma^{-1}dv)} \\
& \qquad \qquad 
\leq  Ct  \| \nabla \tilde \theta \|^2_{L^2(dx)}   \| \cL^{-1} \psi\| ^2_{L^2 (Ma dv)} 
+\frac1{2\alpha^2} \int_0^t \int M \, \tilde g_\alpha \cL  \tilde g_\alpha (s,x,v) dxdvds
\end{align*}
where in the last inequality we have used the fact that for any function~$h$,
$$
\begin{aligned}
\|  \cL h\|^2_{L^2 (Ma^{-1}dv)} & = \int  \Big( \int \big(h(v_*)+h(v) - h(v'_*) - h(v') 
\big) M(v_*) b dv_*d\omega
\Big)^2 M a^{-1}dv\\
& \leq  \int  \left( \int \big(h(v_*)+h(v) - h(v'_*) - h(v') 
\big)^2M(v_*) b dv_*d\omega
\right) \; \Big( \int M(v_*) b dv_*d\omega
\Big)  M a^{-1}dv\\
& =  \int \int \big(h(v_*)+h(v) - h(v'_*) - h(v') 
\big)^2 \, M \, M(v_*) \, b dv_*d\omega = 4 \int Mh \cL h dv \, .
\end{aligned}
$$
The result follows.
\end{proof}

\begin{rmk}
In the nonlinear case, the counterpart of the energy inequality \eqref{Lentropy} is the entropy inequality
$$
\begin{aligned}
\frac1{\alpha^{2\gamma}} &\int h\Big( {f_\alpha \over M} \Big) (t,x,v) Mdvdx+\frac1{\alpha^{2+2\gamma}} \int_0^t \int D(f_\alpha )(s,x) dsdx \\
&\leq \frac1{\alpha^{2\gamma}} \int h\Big( {f_{\alpha,0} \over M} \Big) (x,v)Mdvdx -\frac1{\alpha^{1+2\gamma}} \int_0^t \int_{\d\Omega \times {\mathbb R}^3} h\Big( {f_\alpha \over M} \Big)(s,x,v) M (v\cdot n (x) ) d\sigma_x dvds
\end{aligned}$$
where  $h(z) = z \log z - z+1$. In order to kill the singularity coming from the boundary term in the right-hand side, a natural idea is to modulate the entropy.  This strategy has been successfully used for stochastic dynamics, see for example \cite{ELS,KLO,FLM}.
Unfortunately, it is not known whether the local conservation of energy holds for the Boltzmann equation, and anyway there is no uniform control on the energy flux. An alternative would then be to introduce a renormalized modulated entropy, but it is not clear how to define it.
\end{rmk}

Once we have uniform bounds on $g_\alpha$ and $q_\alpha (v,v_*) := \displaystyle \frac1\alpha \big(g_\alpha(v) +g_{\alpha } (v_*) - g _\alpha(v')  - g _{\alpha  }(v'_*)\big )$, there is no particular difficulty to prove that any joint limit points $(\bar g,\bar q)$  satisfy
\begin{equation}
\label{kinetic-lim}
 \cL \bar g = 0 \, , \qquad v\cdot \nabla_x \bar g = - \int \bar q (v,v_*) M(v_*)  b(v-v_*, \omega) dv_*d\omega
 \end{equation}
in the sense of distributions. Then, taking limits in the conservations of mass, momentum and energy, we obtain that the moments of $\bar g$ satisfy the incompressibility and Boussinesq constraints as well as
$$ \d_t u -\nu \Delta _x u +\nabla_x p = 0\, ,  \quad \d_t \theta -\kappa \Delta _x \theta = 0\,.$$

\subsection{Second step of the proof~: analyzing the boundary conditions}

The second difficulty is to obtain the limiting boundary conditions.
>From the control on the boundary term in the modulated entropy inequality, we deduce easily that 
$$  \tilde g_\alpha - \int \tilde g_\alpha d\mu_x (v) \to 0\,  \hbox{ in }\,L^2 (dtd\mu_x(v)d\sigma_x )\,.$$
It remains then to prove that  \emph{the limit of the trace is the trace of the limit}, which will imply that 
$$\bar g  -\bar \theta_{|\d \Omega}  {|v|^2-  (d+2) \over 2} - \Big(   \rho- \frac12 \tilde \theta\Big)  = 0\,  \hbox{ on }\, \Sigma_+\,,$$
and consequently that 
$$ u = 0 ,\quad \theta=\bar  \theta_{|\d \Omega} \,\hbox{ on } \,\d \Omega\,.$$
We shall follow essentially the same path as in \cite{MSR} (see also \cite{SR-LNM}).

\begin{prop}
With the previous notation, 
$$ g_{\alpha | \Sigma} \rightharpoonup g_{|\Sigma} \hbox{ weakly in } L^2 \big((1+|v|^2)^{-1} dt d\mu_x(v) d\sigma_x\big)\, .
$$
\end{prop}

\begin{proof}
The key point is that, in viscous regimes, the energy/entropy dissipation provides a uniform bound on the transport
$$ (\alpha \d_t +v\cdot \nabla_x) g_\alpha = -\frac1\alpha \cL g_\alpha  \in L^2 ( dt dx M a^{-1} dv ) \,.$$
We then expect to have 
\begin{itemize}
\item some regularity of the moments by the averaging lemma (see \cite{GLPS})
$$\left\|\int g_\alpha M\varphi(v) dv \right\|_{ L^2_t (H^{1/2}_x)} \leq C\,.$$ 
\item weak traces defined by  Green's formula (see \cite{bardos})
\begin{equation}
\label{trace-green} 
\begin{aligned}
\int_{t_1}^{t_2} \int_{\d \Omega} \varphi g_\alpha &(v\cdot n ) Mdv d\sigma_x dt = \alpha \int_\Omega \int \varphi  g_\alpha (t_1) Mdvdx - \alpha \int_\Omega \int \varphi  g_\alpha (t_2) Mdvdx\\
& +\int_{t_1}^{t_2} \int_ \Omega (v\cdot \nabla_x \varphi) g_\alpha Mdvdxdt -\int_{t_1}^{t_2} \int _\Omega \int \varphi  q_\alpha  b(v-v_*,\omega) M(v_*)   Md\omega dv_*dvdxdt,
\end{aligned}
\end{equation}
with $q_\alpha (v,v_*) := \displaystyle \frac1\alpha \big(g_\alpha(v) +g_{\alpha } (v_*) - g _\alpha(v')  - g _{\alpha  }(v'_*)\big )$.
\end{itemize}

\medskip

The first step is to establish a uniform a priori bound on the traces. From the identity
$$
\begin{aligned}
\int_{t_1}^{t_2} \int_{\d \Omega} g_\alpha^2 {(v\cdot n(x))^2 \over 1+|v|^2} \chi(x) Mdv d\sigma_x dt 
&= \alpha \int_\Omega \int   g^2_\alpha (t_1){v\cdot n(x) \over 1+|v|^2} \chi(x)  Mdvdx 
- \alpha \int_\Omega \int  g^2_\alpha (t_2) { v\cdot n(x) \over 1+|v|^2} \chi(x)  Mdvdx\\
& 
 + \int_{t_1}^{t_2} \int_\Omega \int   g^2_\alpha (t){1 \over 1+|v|^2}
v \cdot \nabla_x \big( (v\cdot n(x)) \chi(x) \big) Mdvdx dt  \\
& -2 \int_{t_1}^{t_2} \int _\Omega \iint g_\alpha q_\alpha {v\cdot n(x) \over 1+|v|^2} b(v-v_*,\omega) \chi(x) M(v_*)M(v)d\omega dv_* dvdxdt\,,
\end{aligned}
$$
where $n$ is a smooth extension of the outward unit normal in a neighborhood of the boundary associated with the test function $\chi$, we deduce that $(g_{\alpha | \d \Omega})$ is weakly compact in $L^2 ((1+|v|^2)^{-1} dt  d\mu_x(v) d\sigma_x)$.

\medskip

The second step is then to identify the limit $\gamma$ of $(g_{\alpha | \d \Omega})$. Taking limits in (\ref{trace-green}), we get
$$\int_{t_1}^{t_2} \int_{\d \Omega} \varphi \gamma (v\cdot n ) Mdv d\sigma_x dt = \int_{t_1}^{t_2} \int_ \Omega (v\cdot \nabla_x \varphi) \bar g Mdvdxdt -\int_{t_1}^{t_2} \int _\Omega  \int \varphi \bar q b(v-v_*,\omega) M(v_*)M(v)dv_*dvdxdt\,.$$
Applying Green's formula to the limiting kinetic equation (\ref{kinetic-lim}), we can express  the right-hand side of this identity in terms of the trace of $\bar g$, from which we deduce that $\gamma = \bar g_{| \d \Omega } $.
\end{proof}

\begin{rmk}
As already noted, the control on the free transport provides some regularity with respect to the spatial variable $x$. This is of course a major ingredient when considering the non linear case with convection terms. In that case, we also need to control the fast time oscillations (acoustic waves) which describe the compressible effects, and to prove that there is no constructive interference of these waves.
\end{rmk}


\section{The case of real rarefied gases}

\subsection{Diffusive limits of the hard-sphere dynamics}
\label{sec: hard sphere gas}

In the previous section, we have proven 
that the Fourier law is a scaling limit of the Boltzmann equation. However it does not really answer our original question of deriving the Fourier law from the first principles of mechanics, since the Boltzmann equation is already 
an idealization of gas dynamics.
What we would like to start with is a microscopic model of $N$ interacting particles. We denote by $X_N:=(x_1,\dots,x_N)$ their positions and by~$V_N:=(v_1,\dots,v_N)$ their velocities.
For the sake of simplicity, we shall focus on the hard-sphere interaction~:
$$
{dx_i \over dt } = v_i\, , \qquad {dv_i \over dt } = 0 $$
on 
$$\cD_N := \{Z_N = (X_N, V_N) \in \Omega^N \times {\mathbb R}^{dN} \,/\, \forall i\neq j ,\, |x_i - x_j|>\eps\}\,.$$
The collision law is given by the elastic reflection condition 
\begin{equation}
\label{reflection}
 v'_i = v_i - (v_i-v_j)\cdot n_{ij} \, n_{ij}, \quad v'_j = v_j +  (v_i-v_j)\cdot n_{ij}\,  n_{ij} \hbox{ when } x_i - x_j= \eps n_{ij}\, , \quad |n_{ij}| = 1\,.
 \end{equation} 
 It remains then to prescribe a boundary condition on $\d \Omega$. The natural way to express that one has a thermostat fixing the temperature $T$,  is to impose a stochastic reflection condition
 with law $M_\Sigma$.

 Although it is certainly not physically relevant (even neglecting quantum effects, we expect the  atoms to have long-range interactions), this model shares many important features with general systems of interacting particles~:
 \begin{itemize}
 \item it is completely deterministic~: given the configuration at time 0, we know exactly the state of the system both in the past and in the future;
 \item it satisfies the Poincar\'e recurrence theorem, meaning that its state is arbitrarily close to any given configuration for arbitrarily large times.
 \end{itemize}
 At first sight, those properties seem incompatible with the Boltzmann dynamics or the Fourier law since those evolutions  are not time-reversible, and they predict a relaxation towards equilibrium.
The suitable way to understand this paradox is to see the Boltzmann dynamics (and its hydrodynamic approximations) as an averaged behaviour, i.e. a \emph{law of large numbers} as the number of particles $N$ tends to infinity.

\medskip
More precisely, instead of considering one specific realization  corresponding to a given initial configuration, we shall endow the space of initial configurations 
$\cD_N$
with a probability density $F_{N,0}$ (assumed to be symmetric as particles are exchangeable), and then look at the evolution of this measure under the hard-sphere dynamics.
We shall further assume that initially the particles are identically distributed and almost independent (up to the exclusion condition), so that 
\[
F_{N,0} (Z_N) = {1_{\cD_N}\over \cZ_N} \prod_{i=1}^N f_0(x_i,v_i) \,,
\]
where the partition function~$\cZ_N$ is a normalization constant ensuring that~$\displaystyle \int F_{N,0} dZ_N = 1$.
The distribution~$F_N\equiv F_N(t,Z_N) $ satisfies the Liouville equation
\begin{equation}
\label{liouville}
\d_t F_N +\sum_{i=1}^N v_i \cdot \nabla_{x_i} F_N = 0 \quad \hbox{on } \cD_N, \qquad F_{N|t = 0} = F_{N,0}\, ,
\end{equation}
supplemented with the boundary condition
$$F_N(Z_N) = F_N(Z_N^{'(i,j)})\quad \hbox{ on } \d \cD_N^{i,j+}$$
 where $Z_N^{'(i,j)}$ is the configuration obtained from $Z_N$ by the scattering (\ref{reflection}) and
 $$
	\begin{aligned}\d\cD_N^{i,j+}:= \Big\{Z_N \in {\mathcal D}_N \, \Big| \quad   &|x_i-x_j| = \eps \, , \quad   (v_i-v_j) \cdot (x_i- x_j) >0 \\
& \mbox{and}  \quad\forall (k,\ell) \in  [1,N]^2\setminus \{(i,j)\} \, , k \neq \ell \, ,     |x_k-x_\ell| > \eps
\Big\} .
\end{aligned}
$$
In the following, we shall consider the idealized situation of a periodic domain $\Omega ={\mathbb T}^d := [0,1]^d$ in order to neglect the boundary effects. 
We are going to review the main steps of the derivation of Fourier law from~\cite{BGSR2} in this simplified framework.

It is possible to define the microscopic counterpart of the  heat reservoirs  \eqref{Req} at the  boundary of a general domain $\Omega$
by adding the diffuse reflection condition
 $$
 F_N(Z_N) = M_\Sigma ( v_i) \int F_N (Z_N^{'(i)}) (v'_i \cdot n_i) _- dv'_i \quad \hbox{ on } \d \cD_N^{i-}
 $$
 where $Z_N^{'(i)}$ is the configuration obtained from $Z_N$ by changing $v_i$ into $v'_i$ and
\begin{align*}
\d  \cD_N^{i-}:= \Big \{Z_N \in {\mathcal D}_N \; \Big|  \quad   d(x_i,\partial\Omega) = \eps  \,  \,\mbox{and}   \, \,
	v_i\cdot n_i<0
	\Big\} \, ,
\end{align*}
denoting~$n_i$ the outward unit normal at the contact point.
Nevertheless, we are not aware of any result on the derivation of  Boltzmann equation from microscopic models with such non-equilibrium boundary conditions.

\medskip

We   expect $F_N$ to be a very complicated object (depending on all individual trajectories $(X_N(t), V_N(t))$  of the hard-sphere dynamics), but we are interested only in the statistical properties of the system, in particular on \emph{the 1-particle density}
 $$F_N^{(1)} (t,x,v) := \int F_N(t,x, v, x_2,v_2, \dots, x_N, v_N)\,  dz_2 \dots dz_N\,,$$
from which we can compute the thermodynamic parameters of the gas $U_N(t,x), T_N(t,x)$.
 The Boltzmann equation for perfect gases (\ref{Beq}) (and its linearized version (\ref{LBeq})) expresses a balance between transport and collisions. When the Knudsen and Mach numbers are of order 1, particles undergo in average one collision per unit of time, while moving of a length $O(1)$. A rough computation due to Maxwell shows that at the microscopic level a tagged particle moving with velocity $v$  in a regular lattice of $N$ fixed obstacles of size $\eps$ (in a unit volume $\Omega \subset {\mathbb R}^d$) will encounter an obstacle on $[0,t]$ with probability $O(N\eps^{d-1} Vt)$. 
 \begin{figure}[h] 
\centering
\includegraphics[width=3in]{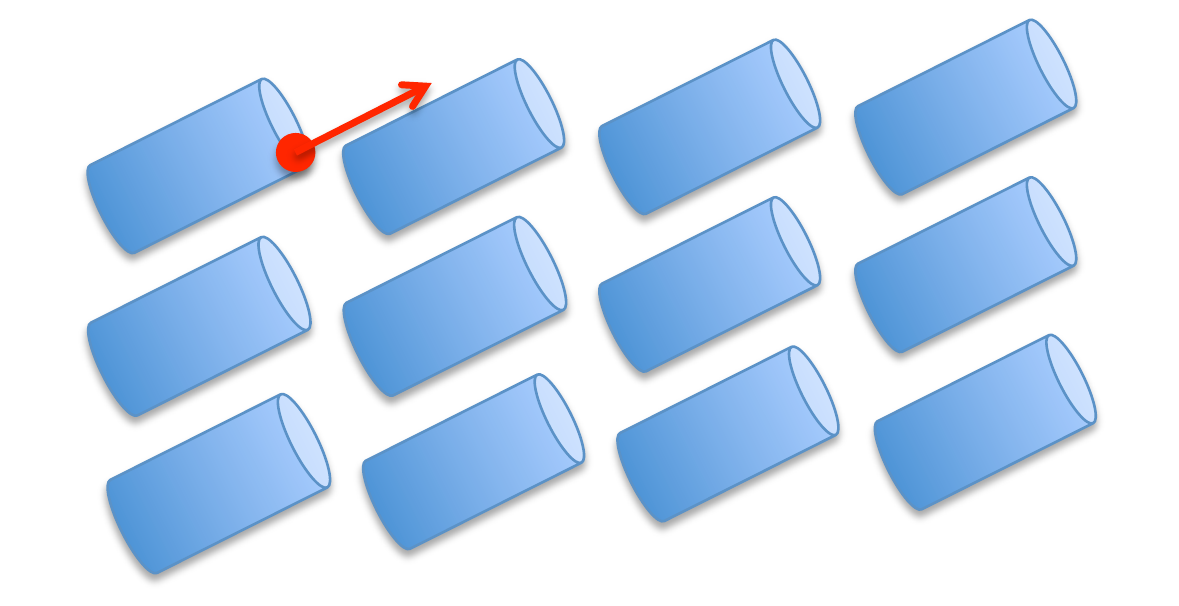} 
\caption{\small  $N$ spheres of size $\eps$ are placed on a lattice. The volume covered by a particle of velocity $v$ during a time $t$ is $|v| t \eps^{d-1}$. Thus the mean free path is of order $1/(N \eps^{d-1})$ which leads to the Boltzmann-Grad scaling $N\eps^{d-1} = O(1/\Kn)$ (see \cite{grad}).
}
\end{figure}
In other words, the perfect gas approximation should be obtained in the limit $N\to \infty$, under the Boltzmann-Grad scaling~$N\eps^{d-1} = O(1/\Kn)$. Recall that because of the very different natures of the hard-sphere dynamics and perfect gas dynamics, we cannot expect a strong convergence in this limit $\eps \to 0$, $N\to \infty$, $N\eps^{d-1} = O(1)$, but only a \emph{weak convergence after averaging over initial configurations.}

\subsection{The heat equation as a scaling limit of the local conservation of energy}

At the formal level, one can see that the first marginal $F_N^{(1)}$ of the $N$-particle distribution $F_N$ satisfies an equation similar to the Boltzmann equation, by integrating (\ref{liouville}) with Green's formula, and using the exchangeability~:
 $$ \d_t F_N^{(1)} + v\cdot \nabla_{x}  F_N^{(1)} = (N-1)\eps^{d-1} \int F_N^{(2)} (t,x, v, x+\eps \omega_2, v_2) ((v_2- v)\cdot \omega_2)  d\omega_2 dv_2\,.$$
 Splitting the boundary term according to the sign of $(v_2- v)\cdot \omega_2$, and using the reflection condition to express the distribution of post-collisional configurations in terms of the distribution of pre-collisional configurations
 $$ F_N^{(2)} (t,x, v, x+\eps \omega_2, v_2)= F_N^{(2)} (t,x, v', x+\eps \omega_2, v'_2) \hbox{ with } v'= v- (v-v_2) \cdot \omega_2\,\omega_2\,, \quad v'_2 = v_2+ (v-v_2) \cdot \omega_2\,\omega_2\,,$$
 we obtain finally
 $$\d_t F_N^{(1)} + v\cdot \nabla_{x}  F_N^{(1)} = (N-1)\eps^{d-1} \int \Big(F_N^{(2)} (t,x, v', x+\eps \omega_2, v'_2) -  F_N^{(2)} (t,x, v, x-\eps \omega_2, v_2)\Big)  ((v_2- v)\cdot \omega_2)_+  d\omega_2 dv_2 
$$
  which we shall abbreviate 
  \begin{equation}
 \label{BBGKY1}
  \d_t F_N^{(1)} + v\cdot \nabla_{x}  F_N^{(1)} = C_{1,2} F_N^{(2)} .
  \end{equation}
 
 At this stage, there are three important differences with the Boltzmann equation
 \begin{itemize}
 \item the collision operator involves a prefactor $(N-1) \eps^{d-1}$ which is not exactly $1/\Kn$. Note however that this would not be the case if we  had used the grand canonical formalism instead of the canonical formalism (i.e. not prescribing a priori the number of particles $N$, but just a Poisson law for this number);

 \item the particles are not pointwise, which implies that there is a small shift between colliding particles as~$|x-x_2| = \eps$;
 
 \item but the main difficulty is that the equation for $F_N^{(1)}$ is not closed~: the probability of having a collision depends on the joint probability $F_N^{(2)}$ of having two particles in a collisional configuration.
 \end{itemize}
This means that in addition to the scaling assumption, the perfect gas approximation relies on a \emph{strong chaos} \emph{property}, referred to as Boltzmann's Stosszahlansatz (which is actually assumed to hold for all times in the original work of Boltzmann). The question here  is to understand when this assumption is satisfied, for the Boltzmann equation and subsequently the Fourier law to provide good approximations of the macroscopic evolution of the particle system.

 \medskip
 The mathematical justification of the Boltzmann equation  in the Boltzmann-Grad limit, {\it in the absence of boundary effects}, goes back to Lanford \cite{lanford}. It relies on the study of the BBGKY hierarchy, i.e. the hierarchy of equations governing the whole family of marginals $(F_N^{(s)})_{s \leq N}$ (since none of them is closed for $s<N$)~:
\begin{equation}
\label{BBGKY}
 \d_t F_N^{(s)} + \sum_{i=1}^s v_i \cdot \nabla_{x_i}  F_N^{(s)}:= C_{s,s+1} F_N^{(s+1)}, 
 \end{equation}
 where $C_{s,s+1}$ has a structure very similar to $C_{1,2}$: defining~$Z_s^{\langle i\rangle}:=(z_1,\dots,z_{i-1},z_{i+1},\dots z_s)$,
 $$
 \begin{aligned}
 C_{s,s+1} F_N^{(s+1)}(t,Z_s) = &(N-s)\eps^{d-1} \sum_{i=1}^s\int \Big(F_N^{(s)} (t,Z_s^{\langle i\rangle}, x_i, v'_i, x_i+\eps \omega_{s+1} , v'_{s+1} ) \\
 & -  F_N^{(s+1)} (t,Z_s^{\langle i\rangle},x_i, v_i, x_i-\eps \omega_{s+1} , v_{s+1})\Big)  ((v_{s+1}- v_i)\cdot \omega_{s+1})_+  d\omega_{s+1} dv_{s+1}\,.
 \end{aligned}
 $$
 In particular it involves a prefactor $(N-s)\eps^{d-1}$ and micro-translations $O(\eps)$, which are expected to disappear in the limit $\eps \to 0$.
 The main difficulty at this stage is that the transport equation is set on $\cD_s$ (with the reflection condition on 
  the internal boundaries $\d\cD_s^{i,j}$).
 Taking limits as $\eps \to 0$ then requires   proving that, in average, this operator behaves as the free transport, or in other words that collisions between  particles which are already correlated are negligible (see \cite{GSRT} or \cite{PSS} for a careful estimate of these recollisions). Lanford's theorem states the convergence towards the Boltzmann equation for short times starting from any chaotic initial data.
 
 \begin{thm}[Lanford] 
 Consider a system of $N$ hard spheres of diameter $\eps $ on ${\mathbb T}^d$, initially ``independent" and identically distributed with density $f_0$
\begin{equation}
\label{eq: Gibbs measure}
F_{N,0}(X_N, V_N) = {\indc_{\cD_N}(X_N) \over \cZ_N} \prod_{i=1}^N f_0 (z_i)  
\hbox{ with } f_0 \hbox{ continuous  such that } \| f_0 (z) \exp (\beta |v|^2 +\mu )\|_\infty \leq 1\,,
\end{equation}
for some constants $\beta >0, \mu$.
 Then, in the Boltzmann Grad limit $N\to \infty$, $\eps \to 0$, $N\eps^{d-1}\alpha = 1$, the 1-particle distribution $F_N^{(1)}$ converges almost everywhere  to the solution of the Boltzmann equation
 $$\d_t f +v\cdot \nabla_x f = \frac1\alpha Q(f,f)\, , \qquad f_{|t = 0} = f_0\, ,$$
 on some time interval $[0, C_{\beta, \mu} \alpha]$.
Furthermore, one has the convergence of all marginals $F_N^{(s)}$ towards the  tensor products $f^{\otimes s}$, which implies that the empirical measure $\displaystyle \mu_n(t) = \frac1N \sum_{i=1}^N \delta_{x-x_i(t), v-v_i(t)}$ converges in law to $f$.
 \end{thm}
 This result is however not completely satisfactory, especially if we have in mind to derive the Fourier law.
 
 \begin{itemize}
 \item First of all, it holds only for short times (short being measured here in terms of the average number of collisions per particle), which is incompatible with the fast relaxation limit (where we assume that the time of observation is much longer than the mean free time, i.e. $\alpha \ll1$).
  
 \item Furthermore, extending this result  to the case of  a bounded domain $\Omega$ with diffuse reflection on the boundary is not an easy exercise~:  given a smooth temperature profile $T$ on $\d \Omega$, the usual Cauchy-Kowalewski estimates for the BBGKY hierarchy fail in general, even for short times.
\end{itemize}
The challenge to obtain the Fourier law directly from the hard sphere system is then to obtain a priori bounds and to establish the propagation of chaos on \emph{time intervals which are much longer than the mean free time}. 
 
 It is important to note here that there is something strange in our mathematical understanding of chaos. In Lanford's strategy, independence is assumed at initial time, and it is proved that it will survive at least for a short time, but collisions are somehow the enemy. What is expected at the physical level is quite the opposite~: collisions should induce a mixing mechanism and bring some chaos. This is however very far from being a mathematical statement!

  \medskip
 A small step in this direction has been made in \cite{BGSR2} in dimension $d=2$ when considering very small fluctuations  around an equilibrium, say for instance the Gibbs measure
 $$ 
 M_N(Z_N) = {\indc_{\cD_N}(Z_N) \over \cZ_N} {1\over (2\pi)^N} \exp \Big( -\sum_{i=1}^N {|v_i|^2\over 2} \Big) \, ,\quad 
 F_{N,0} = M_N +\frac1N \delta \! F_{N,0}\, .
$$
{\color{black} Note that the existence of such a microscopic equilibrium is the only way we are able to take advantage of the cancellations between the gain term and loss term in the collision operator. The usual restriction on the  time of convergence is related to the fact that these cancellations are neglected when establishing a priori bounds.}

{\color{black} Close to equilibrium, we can actually use some comparison principle to improve these a priori bounds.}
More precisely, the main idea is that we can control a very weak form of chaos by an argument of 
 exchangeability, independently of the dynamics of the system  (see Lemma \ref{lem: exchangeability}).
  We can then import these (global in time) a priori bounds in the proof of Lanford, to extend for long times the convergence towards the (linearized) Boltzmann equation.

 \begin{thm}[Bodineau, Gallagher, Saint Raymond] 
 \label{bgsr-thm}
Consider a system of $N$ hard spheres of diameter $\eps $ on ${\mathbb T}^2$, initially close to equilibrium
\begin{align*}
 \delta \! F_{N,0} (Z_N) = M_{N}(Z_N)  \sum_{i=1}^N   g_0(z_i)  \hbox{  with } \int Mg_0(z) dz =0\, , \quad \text{and} \quad
\| g_0\|_{L^\infty}   \leq C_0\, .
\end{align*}
Then  in the Boltzmann Grad limit $N\to \infty$, $\eps \to 0$, $N\eps \alpha = 1$, the one-particle fluctuation~$\delta \! F_N^{(1)}(t, x,v)$ 
is close to  the solution of the linearized Boltzmann equation $M g_\alpha   $
$$\d_t g_\alpha +v\cdot \nabla_x g_\alpha = -\frac1\alpha \cL g_\alpha\, , \quad g_{\alpha | t = 0} = g_0\,,
$$
in the sense that for $t\leq  T/\alpha$,
$$\big\|  \delta \! F_N^{(1)}(t)- M g_\alpha (t)
\big\|_{L^2} \leq C   
 \frac{  e^{C/  \alpha^2} } { \log\log N} \, \cdotp
$$
 \end{thm}
In particular, in diffusive regime, we recover the Fourier law.
 \begin{corollary}
\label{stokes-cor}
Consider a system of $N$ hard spheres of diameter $\eps $ on ${\mathbb T}^2$, initially distributed according to
$$
 \delta \! F_{N,0} (Z_N) = M_{N}(Z_N)  \sum_{i=1}^N  \Big(  u_0(x_i) \cdot v _i+ \frac{|v_i|^2 - 4}2 \, \theta_0(x_i)\Big)    \hbox{  with }\nabla_x \cdot u_0 = 0\, .
 $$
Then  in the Boltzmann Grad limit $N\to \infty$, $\eps \to 0$, $N\eps = \frac1 \alpha \to \infty$ (with $\alpha (\log\log \log N)^{1/2}\gg1$), the  one-particle fluctuation~$\delta \! F_N^{(1)}({ \tau\over \alpha} ,x,v)$ 
is close (in $L^2$ norm) to  $M g(\tau,x,v):=  M\big( u (\tau,x) \cdot v + \frac{|v|^2 - 4}2 \theta (\tau,x) \big)\, ,$
where~$(u,\theta)$ satisfies the Stokes-Fourier equations
$$\begin{aligned}
&\partial_\tau u -\nu  \Delta_x u = -\nabla_x p  \,, \quad \nabla_x \cdot u = 0\,, \quad u_{|t = 0} = u_0\,,\\
&\partial_\tau \theta -\kappa  \Delta_x \theta = 0\,, \quad \theta_{|t = 0} = \theta_0\,. \\
\end{aligned}
$$
 \end{corollary}
Note that the Knudsen number $\alpha$, although small, is much larger than the radius of particles $\eps$
$$\alpha \gg (\log \log|\log \eps| )^{-1/2}$$
and that such a scale separation does not really make sense from the physical point of view. 
But, for the moment, we do not know how to generalize these results in the nonlinear setting~:  this is  due in particular to the fact that the entropy (which is the natural uniform bound) is not suitable to control quadratic nonlinearities (which is also the reason why the Boltzmann equation is only known to have  renormalized solutions globally in time).

%

\subsection{Lanford's proof}

We shall not give in this note all the details of the proof of Theorem \ref{bgsr-thm} as it  is rather technical. 
The general strategy follows the lines of Lanford's proof. More precisely, it relies on an integral formulation of the BBGKY hierarchy (\ref{BBGKY}) obtained by iterating  Duhamel's formula
$$\delta \! F_N^{(1)}(t) = \sum_{n = 1}^{N}  Q_{1,n+1} (t) \, \delta \! F_{N,0} ^{(n +1)}\, ,
$$
where $Q_{1,n+1}$ encodes $n$ collisions and ${\bf S}_s$ is the transport operator on $\cD_s$
$$Q_{1,n+1} (t) :=    \int_0^t \int_0^{t_2}  \dots \int_0^{t_{n}}  {\bf S}_1(t-t_2) C_{1,2}  {\bf S}_{2}(t_2-t_{3})  
  C_{2,3} \dots   {\bf S}_{1+n}(t_{1+n})   \, dt_{2} \dots dt_{n+1} \,.$$
  This expansion has  a geometric interpretation  in terms of collision trees $(a_i) _{1\leq i \leq n+1} $ with $a_i <i$ (encoding the collisional particle in $C_{i-1,i}$), and pseudo-trajectories defined as follows.
  
\begin{e-definition}
[Pseudo-trajectory]
\label{pseudotrajectory}
Given~$z_1 \in {\mathbb T}^2\times {\mathbb R}^2$, $t>0$ and a collision tree~$a $, consider a collection of times, angles and velocities~$(T_{2,n+1}, \Omega_{2,n+1}, V_{2,n+1}) = (t_i, \omega_i, v_i)_{2\leq i\leq n+1}$ with~$0\leq t_{n+1}\leq\dots\leq t_2\leq t$. We then define recursively the pseudo-trajectories of the  backward BBGKY dynamics as follows
\begin{itemize}
\item in between the  collision times~$t_i$ and~$t_{i+1}$   the particles follow the~$i$-particle backward flow on $\cD_i$ (with specular reflection on $\d \cD_i$);
\item at time~$t_i^+$,  particle   $i$ is adjoined to particle $a_i$ at position~$x_{a_i}(t_i^+) + \eps \omega_i$ and 
with velocity~$v_i$, provided~$|x_i-x_j(t_i^+)| > \eps$ for all~$j < i$ with~$ j \neq a_i$.  
If $(v_i - v_{a_i} (t_i^+)) \cdot \omega_i >0$, velocities at time $t_i^-$ are given by the scattering laws
$$\begin{aligned}
v_{a_i}(t^-_i) &= v_{a_i}(t_i^+) - (v_{a_i}(t_i^+)-v_i) \cdot \omega_i \, \omega_i  \, ,\quad 
v_i(t^-_i) &= v_i+ (v_{a_i}(t_i^+)-v_i) \cdot \omega_i \,  \omega_i \, .
\end{aligned}
$$
\end{itemize}
 We denote  by~$z_i(a, T_{2,n+1}, \Omega_{2,n+1}, V_{2,n+1},\tau)$ the position and velocity of the particle labeled~$i$, at time~$\tau\leq  t_i$. The  configuration obtained at the end of the tree, i.e. at time 0, is~$Z_{n+1}(a, T_{2,n+1}, \Omega_{2,n+1}, V_{2,n+1},0)$. The set of admissible parameters (i.e. such that the dynamics has no overlap and the pseudo-trajectory is defined up to time 0) is denoted $G_n(a)$.
\end{e-definition}
We therefore end up with the following representation formula
$$
\begin{aligned}
 \delta \! F^{(1)} _N(t) =\sum_{n=0}^{N-1}   (N-1) \dots\big(N-n\big) \eps^{n} \sum_{a \in \cA_n}   \int_{G_n(a)}& dT_{2,n+1} d\Omega_{2,n+1}    dV_{2,n+1}\Big( \prod_{i=2}^{n+1}  \big( (v_i -v_{a_i} (t_i)) \cdot \omega_i \Big)  
 \\
 &\quad  \times 
\delta \! F_{N,0}^{(n+1)}\big (Z_{n+1}(a, T_{2,n+1}, \Omega_{2,n+1}, V_{2,n+1},0)\big)  \, .
 \end{aligned}
 $$
Note that the factor $(N-1) \dots\big(N-n\big) \eps^{n}$ is of order $\alpha^n$ in the Boltzmann-Grad limit.

\medskip

Lanford's proof can then be decomposed in two steps~:
\begin{itemize}
\item we first compare the pseudo-trajectories of the BBGKY hierarchy with the pseudo-trajectories describing the Boltzmann dynamics. We expect these two families of pseudo-trajectories to differ because of the small shifts at the collision times ($|x_i- x_{a_i}| = \eps$), and because of the possible recollisions due to the fact that ${\bf S}_s$ is the transport with exclusion (in $\cD_s$). Note that they will also have slightly different weights in the representation formula because of the prefactors $(N-s)\eps$ (replaced by $1/\alpha$ in the limit).
\item we then use an argument of dominated convergence to sum over all possible collision trees, and obtain the convergence of the series.
\end{itemize}
We stress the fact that  only the second step imposes a restriction on the time of convergence in the original proof of Lanford. It is indeed based on (loss) continuity estimates for the collision operators in weighted $L^\infty$ spaces  (with exponential decay in $v$)  denoted by $\tilde L^\infty$ 
$$ \Big\| \int_0^t d t_1 {\bf S}_s (t - t_1 ) C_{s,s+1} {\bf S}_{s+1} (t_1) f_{s+1} \Big\|_{\tilde L^\infty} 
 \leq C s  { t\over \alpha}  \| f_{s+1}\|_{\tilde L^\infty} \, ,
$$
 which provide
\begin{equation}
\label{infty-est}
\| Q_{1,n+1} (t) F_{N,0} ^{(n +1)} \|_{\tilde L^\infty}  \leq \Big(C_{\beta,\mu}  { t\over \alpha} \Big)^n \,.  
\end{equation}
Extending the convergence beyond Lanford's time relies on two important ideas~:
\begin{itemize}
\item a pruning procedure to sort out pathological pseudo-trajectories involving too many collisions on small time intervals;
\item   refined continuity estimates for the collision operators which allow to control these bad terms using cumulant expansions.
\end{itemize}
The technical arranging of these arguments is rather complicated but we shall just describe roughly the main ideas.

\subsection{Extending the time of validity of Lanford's series expansion}

The \emph{pruning procedure} is a strategy which was devised in \cite{BGSR1} in order  to control the growth of collision trees.
The idea is to introduce a sampling in time with a (small) parameter $h>0$.
Let~$\{n_k\}_{k \geq 1}$ be a  sequence of integers, typically~$n_k = 2^k$.
We   study the dynamics up to time  $t= \tau/  \alpha$ by splitting the time interval~$[0,t]$  into~$K$ intervals of size~$h$,  and controlling the number of collisions  on each interval. 
In order to discard  trajectories with a large number of collisions in the iterated Duhamel formula, we define collision trees ``of controlled size" by the condition that they have strictly less than $n_k$ branch points on the  interval~$[t-kh ,t-(k-1) h]$.
Note that by construction, the trees are actually  followed ``backwards", from time~$t$ (large)   to time~$0$. So
we decompose the iterated Duhamel formula, by writing
\begin{equation}
\label{series-expansion}
\begin{aligned}
&\delta \!  F_N^{(1)} (t)    : =   \sum_{j_1=0}^{n_1-1}\! \!   \dots \!  \! \sum_{j_K=0}^{n_K-1} Q_{1,J_1} (h )Q_{J_1,J_2} (h )
 \dots  Q_{J_{K-1},J_K} (h )     \delta \!   F^{(J_K)}_{N,0}   \\
&\qquad  +\sum_{k=1}^K \; \sum_{j_1=0}^{n_1-1} \! \! \dots \! \! \sum_{j_{k-1}=0}^{n_{k-1}-1}\sum_{j_k \geq n_k} \; 
 Q_{1,J_1} (h ) \dots  Q_{J_{k-2},J_{k-1}} (h ) \,  Q_{J_{k-1},J_{k}} (h)    \delta \!   F^{(J_{k})}_N(t-kh) 
    \, ,
 \end{aligned}
\end{equation}
with~$J_0:=1$,~$J_k :=1+ j_1 + \dots +j_k$. The first term on the right-hand side corresponds to the smallest trees (having less than $N_k = n_1+\dots +n_K = O(2^{K+1})$ branching points),  and the second term is the remainder:  it represents trees with super exponential branching, i.e. having at least~$n_k$ collisions during the last time lapse, of size~$h$. 
We have  then to 
prove that  the remainder is small, even for large~$t$ (but small~$h$).

In \cite{BGSR1}, we were interested in the motion of one tagged particle in a gas  close to equilibrium, and we had the global uniform estimate
$$ \forall n, \forall t>0, \qquad  \delta \! F^{(n)}_N  (t, Z_n)  \leq C_0 M_N^{(n)} \leq C_0 C^n M^{\otimes n}$$
which results from the $L^\infty$ bound on the initial data, the maximum principle for the Liouville equation and  
 the fact that the Gibbs measure $M_N$  is invariant under the flow. By (\ref{infty-est}), we thus obtained that the remainder is controlled by
 $$
 \begin{aligned}
  \| \sum_{k=1}^K &\; \sum_{j_1=0}^{n_1-1} \! \! \dots \! \! \sum_{j_{k-1}=0}^{n_{k-1}-1}\sum_{j_k \geq n_k} \; 
 Q_{1,J_1} (h ) \dots  Q_{J_{k-2},J_{k-1}} (h ) \,  Q_{J_{k-1},J_{k}} (h)    \delta \!  F^{(J_{k})}_N(t-kh)  \| _{L^\infty} \\
 &\leq  \sum_{k=1}^K \; \sum_{j_1=0}^{n_1-1} \! \! \dots \! \! \sum_{j_{k-1}=0}^{n_{k-1}-1}\sum_{j_k \geq n_k} C_0 \Big(C { t\over \alpha}\Big) ^{N_{k-1}} \Big(C { h\over \alpha}\Big)^{n_k} \\
 & \leq C_0 \sum_{k=1}^K 2^{k^2} \Big(C^2 { th\over \alpha^2}\Big)^{2^k} 
 \end{aligned}
 $$
 which is small provided that  $ ht \ll\alpha^2 $.
However, under the assumptions of Theorem \ref{bgsr-thm},  the initial datum~$\delta \! F_N$ is no longer~$O(1)$ in~$L^\infty$, and we cannot apply the same strategy.

\subsection{Taking advantage of the exchangeability}

The second key argument   is specific to the linearized setting. 
In dimension $d=2$, the partition function 
introduced in \eqref{eq: Gibbs measure} is bounded from below uniformly in~$N$
\begin{equation}
\label{ZN}
 \cZ_N =\int  \indc_{\cD_N} dZ_N \geq c e^{- C/\alpha^2}\,.
 \end{equation}
This means that  the product  measure $M^{\otimes N}$ and the  measure $M_N$ conditioned by the exclusion constraint are comparable in a very strong sense. This fact will be crucial in the decomposition  \eqref{fNs-exp} and in \eqref{L2-bound} in order to replace $M_N$ by $M^{\otimes N}$. 
 
Combining  estimate \eqref{ZN} with the initial $L^2$ bound (which is the counterpart in this setting of the initial entropy control) 
$$ 
\int {\delta \! F_{N,0} ^2 \over M_N } dZ_N 
 \leq   C  \int  M^{\otimes N} 
\left( \sum_{i=1}^N   g_{0}(z_i) \right)^2  dZ_N  
 \leq   CN\| g_{0}\| _{L^2(Mdvdx)}^2    
$$
and   the conservation law for the Liouville equation, we obtain that, for all $t$
\begin{equation}
\label{L2-bound}
\int {\delta \! F_{N} ^2(t) \over M^{\otimes N} }dZ_N 
 \leq   C e^{C/\alpha^2}\ \int {\delta \! F_{N} ^2 (t) \over M_N }dZ_N 
 \leq   Ce^{C/\alpha^2} \int {\delta \! F_{N,0} ^2 \over M_N } dZ_N  = O(Ne^{C/\alpha^2})\, .
\end{equation}
The following cumulant estimates can then be obtained from this a priori bound.
\begin{lem}
\label{lem: exchangeability}
Let~$\delta \! F_N$ be a mean free, symmetric function such that $\delta \! F_N/ M^{\otimes N} \in L^2( M^{\otimes N} dZ_N )$. 
There exist symmetric functions $f_N^m$  on $\cD_m$ for $1\leq m\leq N$  such that
for all $s \leq N$,  the marginal of order~$s$ has the following cumulant expansion~:
\begin{equation}
\label{fNs-exp}
\delta \! F_N^{(s)}( Z_s) = M^{\otimes s} (V_s) \sum_{m=1}^s  \sum _{\sigma\in {\mathfrak S}^m_s } f_N^{{m}} (Z_\sigma)\,,
\end{equation}
where ${\mathfrak S}^m_s$ denotes the set of all parts of $\{1,\dots ,s\}$ with $m$ elements, and $ \binom{s}{m}  $  is its cardinal. Moreover we have the following decay estimate on the cumulants~:
\begin{equation}
\label{cumulant-est}
\sum _{m=1}^N  \binom{N}{m}\| f_N^{m}  \|^2_{ L^2 (M^{\otimes m} dZ_m)} 
\leq 
  \|\delta \! F_N/M^{\otimes N}\|^2_{L^2 (M^{\otimes N} dZ_N)}  = O(N e^{C/\alpha^2})\,.
\end{equation}
\end{lem}
The upper bound \eqref{cumulant-est} shows that the correlations decrease in~$L^2$-norm according to the number of particles. 
In particular, this 
 implies that  the leading order term, in $L^2$, for the marginals \eqref{fNs-exp} 
comes from the first cumulant
\begin{equation}
\label{fNs-exp 2}
\delta \! F_N^{(s)}( Z_s) = M^{\otimes s} (V_s) \;
\left[ \sum_{m=1}^s   f_N^{{1}} (z_m) \right] + O \left(  \frac{1}{\sqrt{N}} \right)_{L^2} .
\end{equation}
This very weak  chaos property can be interpreted as a step towards proving local equilibrium.


\begin{rmk}
The approximation \eqref{fNs-exp 2} is not strong enough to deduce directly that the equation on the first marginal {\rm(\ref{BBGKY1})} can be closed, i.e. to replace $F_N^{(2)}$ by $(F_N^{(1)})^{\otimes 2}$,  because  the collision operator $C_{1,2}$ is too   singular (see \cite{denlinger, BGSRS} for a discussion on this point).  

Note in addition that the argument cannot be applied directly in dimension greater than 2, since  the Gibbs measure $M_N$ is very different from the tensor product $M^{\otimes N}$ (the partition function is no longer uniformly bounded from below as in \eqref{ZN}).
\end{rmk}

 The difficulty is then to use these $L^2$ a priori bounds to obtain a control on the remainders in the BBGKY series expansion (\ref{series-expansion}). The operator  $C_{s,s+1}$ is indeed  ill-defined in $L^2$ (as it acts on hypersurfaces) and it has to be combined with the transport operator to recover the missing dimension (see \cite{GSRT}~Section~5).
One can show that there exists weighted~$L^2$-norms denoted by~$\tilde L^2$ such that 
\begin{eqnarray}
\label{L2-cont}
\left\| \int_0^t d\tau C_{s,s+1} {\bf S}_{s+1}^0 (-\tau ) h_{s+1} \right\|_{\tilde L^2} 
\leq C  \sqrt{ \frac{s \, t}{ \eps} } \,  \| h_{s+1}\|_{\tilde L^2} \, .
\end{eqnarray} 

\medskip

The constants in the $L^2$-estimate \eqref{L2-cont} are much worse than in the  $L^\infty$-estimate
\eqref{infty-est} as each collision operator induces a loss of order $\eps^{-1/2}$.
Nevertheless estimate (\ref{cumulant-est}) implies that higher order cumulants decay as 
 $f_N^m = O(\eps^{(m-1)/2}) _{L^2}$.
Thus combining (\ref{cumulant-est}) and the singular continuity estimate~(\ref{L2-cont}), we see that the powers of $\eps$ exactly compensate; leading to  a control of the remainder in (\ref{series-expansion}) of the type
\begin{align*}
 \| \sum_{k=1}^K &\; \sum_{j_1=0}^{n_1-1} \! \! \dots \! \! \sum_{j_{k-1}=0}^{n_{k-1}-1}\sum_{j_k \geq n_k} \; 
 Q_{1,J_1} (h ) \dots  Q_{J_{k-2},J_{k-1}} (h ) \,  Q_{J_{k-1},J_{k}} (h)    \delta \!  F^{(J_{k})}_N(t-kh)  \| _{L^2} \\
 &\leq  e^{C/\alpha^2} \sum_{k=1}^K \; \sum_{j_1=0}^{n_1-1} \! \! \dots \! \! \sum_{j_{k-1}=0}^{n_{k-1}-1}\sum_{j_k \geq n_k} C_0 \Big(C { t\over \alpha}\Big) ^{3N_{k-1}/2} \Big(C { h\over \alpha}\Big)^{n_k/2 } \\
 & \leq C_0 e^{C/\alpha^2} \sum_{k=1}^K 2^{k^2} \Big(C^2 { t^{3/2} h^{1/2} \over \alpha^2}\Big)^{2^k} \, .
\end{align*} 
 This can be made as small as needed by choosing $t^{3/2} h^{1/2} \ll\alpha^2$.

In fact, the previous estimates only holds if the total number of recollisions is controlled.
It remains then to control super exponential trees with many recollisions. We shall skip here this part of the proof which requires rather tough geometric estimates on multiple recollisions.


\section{Concluding remarks}

The results presented in this note provide very preliminary answers to the problem of deriving the Fourier law from microscopic mechanical systems.
We  discuss here   the possibility of extending these results to more physically relevant situations, and the blocking points.

The first issue is that, in the case of the hard-sphere dynamics,  boundary conditions introduced Section~\ref{sec: hard sphere gas} have not  yet been taken into account. We however expect that Theorem \ref{bgsr-thm} and Corollary \ref{stokes-cor} can be extended to thermostated systems very close to  equilibrium.

Then the main restriction we have imposed here is a (very stringent) \emph{smallness condition on the  size of the non-equilibrium perturbation}. This condition takes different forms in the case of a perfect gas and a hard-sphere gas,
but it allows to overcome two difficulties :
\begin{itemize}
\item close to equilibrium, the \emph{collision process} is very well controlled : it is both efficient (inducing a relaxation towards thermodynamic equilibrium in a very homogeneous way) and non singular (no correlation structure violating the chaos assumption will appear);
\item close to equilibrium, the \emph{hypoellipticity of the transport} is well understood~: we then expect the thermodynamic fields to be rather smooth and thus to have good homogenization properties.
\end{itemize}

A natural question is to know whether the smallness condition is purely technical or whether the physics beyond this threshold is expected to be more complicated. 
For the particle system, the scaling condition~$f_0 = M(1+\frac1N g_0)$ can be related to the fluctuation field which is controlled  by the central limit theorem. 
 At this scale, the Fourier law describes a dissipation mechanism associated with the mixing property of the time correlations of the fluctuation field (at equilibrium).
Looking at the formal asymptotics, it is not clear that there is any obstruction (other than technical) to get the Fourier law in more general situations, at least in incompressible regimes.




\bibliographystyle{plain}
\bibliography{biblio.bib}

\end{document}